\documentclass[10pt,reqno]{amsart}

\usepackage{amssymb}
\usepackage{amsmath}
\usepackage{amsfonts}
\usepackage{amscd}
\usepackage{bm}
\usepackage{mathrsfs}
\usepackage{color}

\newtheorem{thm}{Theorem}
\newtheorem{cor}[thm]{Corollary}

\newtheorem{prop}[thm]{Proposition}

\newcommand{\n}[1]{\vert#1\vert}
\newcommand{\nm}[1]{\Vert#1\Vert}
\newcommand{\pr}[1]{(#1)}
\newcommand{\br}[1]{[#1]}
\newcommand{\bre}[1]{\{#1\}}

\newcommand{\Bdy}{\partial\bm{\Omega}}
\newcommand{\BR}{\mathbf{B}_R}
\newcommand{\BRc}{\mathbf{B}^c_R}
\newcommand{\BRo}{\mathbf{B}_{R_0}}
\newcommand{\BRco}{\mathbf{B}^c_{R_0}}
\newcommand{\Om}{\bm{\Omega}}
\newcommand{\R}{\mathbf{R}}
\newcommand{\Rn}{\mathbf{R}^n}
\newcommand{\IBR}{\int_{\mathbf{B}_R}}
\newcommand{\IBRc}{\int_{\mathbf{B}^c_R}}
\newcommand{\IOm}{\int_{\Om}}
\newcommand{\IIOm}{\int_{\Om}\int_{\Om}}
\newcommand{\IRn}{\int_{\Rn}}
\newcommand{\IIRn}{\int_{\Rn}\int_{\Rn}}
\newcommand{\dt}{\,\mathbf{d}t}
\newcommand{\dx}{\,\mathbf{d}x}
\newcommand{\dxy}{\,\mathbf{d}x\mathbf{d}y}
\newcommand{\on}{\bm{\omega}_{n-1}}
\newcommand{\Vn}{\mathbf{V}_n}
\newcommand{\Wc}{\mathbf{W}_{\epsilon}}

\newcommand{\ClOm}{C^1(\Om)}
\newcommand{\CclOm}{C^1_c(\Om)}
\newcommand{\CclRn}{C^1_c(\Rn)}
\newcommand{\CosOm}{C^{0,s}(\Om)}
\newcommand{\CogOm}{C^{0,\gamma}(\Om)}
\newcommand{\CosRn}{C^{0,s}(\Rn)}
\newcommand{\CogcRn}{C^{0,\gamma^*_s}(\Rn)}
\newcommand{\DspRn}{D^{s,p}(\Rn)}
\newcommand{\DstRn}{D^{s,2}(\Rn)}
\newcommand{\LiOm}{L^{\infty}(\Om)}
\newcommand{\LlOm}{L^1(\Om)}
\newcommand{\LrOm}{L^r(\Om)}
\newcommand{\LtOm}{L^t(\Om)}
\newcommand{\LiRn}{L^{\infty}(\Rn)}
\newcommand{\LlRn}{L^1(\Rn)}
\newcommand{\LpRn}{L^p(\Rn)}
\newcommand{\LqRn}{L^q(\Rn)}
\newcommand{\LrRn}{L^r(\Rn)}
\newcommand{\LlocRn}{L^1_{loc}(\Rn)}
\newcommand{\LpcRn}{L^{p^*_s}(\Rn)}
\newcommand{\LqVRn}{L^q_{\mathrm{V}}(\Rn)}
\newcommand{\LrKRn}{L^r_{\mathrm{K}}(\Rn)}

\newcommand{\LtKRn}{L^2_{\mathrm{K}}(\Rn)}
\newcommand{\LtKtRn}{L^2_{\tilde{\mathrm{K}}}(\Rn)}
\newcommand{\LtlocRn}{L^t_{loc}(\Rn)}
\newcommand{\MsqpOm}{M^{s;q,p}(\Om)}
\newcommand{\MspOm}{M^{s;p,p}(\Om)}
\newcommand{\MqpRn}{M^{q,p}(\Rn)}
\newcommand{\MsqpRn}{M^{s;q,p}(\Rn)}
\newcommand{\MVsqpRn}{M^{s;q,p}_{\mathrm{V}}(\Rn)}
\newcommand{\MsqpBR}{M^{s;q,p}(\BR)}
\newcommand{\MVsqpBR}{M^{s;q,p}_{\mathrm{V}}(\BR)}
\newcommand{\MsipOm}{M^{s;\infty,p}(\Om)}
\newcommand{\MsqiOm}{M^{s;q,\infty}(\Om)}
\newcommand{\MspcRn}{M^{s;p^*_s,p}(\Rn)}
\newcommand{\MsipRn}{M^{s;\infty,p}(\Rn)}
\newcommand{\MsqiRn}{M^{s;q,\infty}(\Rn)}
\newcommand{\WspOm}{W^{s,p}(\Om)}
\newcommand{\WospOm}{\tilde{W}^{s,p}_0(\Om)}

\newcommand{\WospBRo}{\tilde{W}^{s,p}_0(\mathbf{B}_{R_0})}
\newcommand{\WlpRn}{W^{1,p}(\Rn)}
\newcommand{\WspRn}{W^{s,p}(\Rn)}
\newcommand{\WVspRn}{W^{s,p}_{\mathrm{V}}(\Rn)}
\newcommand{\WVstRn}{W^{s,2}_{\mathrm{V}}(\Rn)}
\newcommand{\WlpqRn}{W^1_{p,q}(\Rn)}
\newcommand{\WspCb}{W^{s,p}(\mathbf{Q})}
\newcommand{\usp}{\frac{\n{u(x)-u(y)}^p}{\n{x-y}^{n+sp}}}
\newcommand{\ulsp}{\frac{\n{u_l(x)-u_l(y)}^p}{\n{x-y}^{n+sp}}}

\begin{document}

\title[Compact embeddings of fractional Sobolev spaces on $\Rn$]
{\bf Compact embeddings of some weighted \\ fractional Sobolev spaces on $\Rn$}

\author[Q. Han]
{\sf Qi Han}

\address{Department of Science and Mathematics, Texas A\&M University at San Antonio
\vskip 2pt San Antonio, Texas 78224, USA \hspace{14.4mm}{\sf Email: qhan@tamusa.edu}}


\dedicatory{Dedicated to my little angel Jacquelyn and her mother, my dear wife, Jingbo.}



\begin{abstract}
  In this paper, we study a family of general fractional Sobolev spaces $\MsqpOm$ when $\Om=\Rn$ or $\Om$ is a bounded domain, having a compact, Lipschitz boundary $\Bdy$, in $\Rn$ for $n\geq2$.
  Among other results, some compact embedding results of $\MVsqpRn\hookrightarrow\LqRn$ and $\MVsqpRn\hookrightarrow\LlRn$ for suitable potential functions $V(x)$ are described.
\end{abstract}

\maketitle

\noindent Recently, fractional Sobolev function spaces $\WspRn$, and the associated nonlocal equations, attracted great attentions due to their impressive applications to various disciplines, as briefed for example in the introduction of Di Nezza, Palatucci and Valdinoci \cite{DPV}.
See also Chapter 4 of Demengel and Demengel \cite{DD} or Chapters 1-2 of Molica Bisci, Radulescu and Servadei \cite{MRS} for some concise, self-contained discussions of this matter.

In his classical monograph \cite[Section 5.1.1]{Ma}, Maz'ya defined the general Sobolev function spaces $\WlpqRn$; using the notation $\MqpRn$, the author described some relationships among those spaces and their compact embedding results in \cite{Ha1,Ha2,Ha3,Ha4} with the desire of providing a complement to Lions \cite[Lemma I.1]{Li} in the sense of function space settings.

In the hope of suggesting a little bit more insight regarding the similarity between $\WlpRn$ and $\WspRn$ as observed in \cite{DD,DPV,MRS}, this paper is devoted to the description of some general fractional Sobolev function spaces $\MsqpRn$ when $n\geq2$, their relations and certain compact embedding results of $\MVsqpRn\hookrightarrow\LqRn$ and $\MVsqpRn\hookrightarrow\LlRn$ for suitable potentials $V(x)\geq0$.
As expected, $\MqpRn$ and $\MsqpRn$ share several common properties.

In the sequel, we always assume that $n\geq2$, $1\leq p,q\leq\infty$, and $0<s<1$.

Denote $\Om$ a bounded domain in $\Rn$ having a compact, Lipschitz boundary $\Bdy$, or $\Om=\Rn$.
Define $\MsqpOm$ to be the Banach space as the completion of the set $\ClOm$ if $\Om$ is bounded, or as that of the set $\CclRn$ for $\Om=\Rn$, with respect to the norm
\begin{equation}\label{Eq1}
\nm{u}_{\MsqpOm}:=\nm{u}_{q,\Om}+\br{u}_{s,p,\Om},
\end{equation}
where $\nm{u}^q_{q,\Om}:=\IOm\n{u}^q\dx$ and $\br{u}^p_{s,p,\Om}:=\IIOm\usp\dxy$.

Note $\WspOm=\MspOm$, and if $\Om$ is bounded, $\MsqiOm=\CosOm$ (independent of $q$).

From now on, we shall write both continuous embedding of function spaces and convergence of functions by `` $\to$ '', compact embedding of function spaces by `` $\hookrightarrow$ '', and weak convergence of functions by `` $\rightharpoonup$ ''.
Other notations will be specified when appropriate.

Below, let's embark on the elaboration of our analyses for the spaces $\MsqpOm$.

\vskip 6pt
\noindent{\bf(I.) $\Om$ is a bounded domain with a compact, Lipschitz boundary $\Bdy$.}
\vskip 6pt

In this case, it is trivial to notice $\MsqpOm\to\LlOm$, so that from the fractional {\sf Poincar\'{e}}'s {\sf inequality} (see for instance Bellido and Mora-Corral \cite[Lemma 3.1]{BM}, or Drelichman and Dur\'{a}n \cite[Estimate (1.3)]{DiDr}), one derives $\MsqpOm\to\WspOm$ in view of {\sf Minkowski}'s {\sf inequality}.
This, together with Corollary 4.53 in \cite{DD}, yields the conclusions listed as follows.

\noindent {\bf(1.)} When $sp<n$, then we have, with $p^*_s:=\frac{np}{n-sp}$,
\begin{equation}\label{Eq2}
\left\{\begin{array}{ll}
\MsqpOm=\WspOm&\mathrm{for}\hspace{2mm}1\leq q\leq p^*_s,\\\\
M^{s;q_1,p}(\Om)\to M^{s;q_2,p}(\Om)&\mathrm{for}\hspace{2mm}p^*_s\leq q_2\leq q_1\leq\infty.
\end{array}\right.
\end{equation}
Also, there exists a constant $C_{p,q}>0$, depending on $n,p,q,s,\Om$, such that
\begin{equation}\label{Eq3}
\nm{u}_{\WspOm}\leq C_{p,q}\nm{u}_{\MsqpOm},\hspace{6mm}\forall\hspace{2mm}u\in\MsqpOm\hspace{2mm}\mathrm{and}\hspace{2mm}q\in\br{1,\infty}.
\end{equation}

\noindent {\bf(2.)} When $sp=n$, then we have
\begin{equation}\label{Eq4}
\MsipOm\to\MsqpOm=\WspOm\hspace{2mm}\mathrm{for}\hspace{2mm}1\leq q<\infty,
\end{equation}
and $\MsqpOm$ is not a subset of $\LiOm$ by \cite[Example 4.25]{DD} unless $q=\infty$.
A recent result of Parini and Ruf \cite{PR} provides a fractional {\sf Moser-Trudinger inequality} that says
\begin{equation}\label{Eq5}
\sup_{\br{u}_{s,p,\Rn}\leq1}\IOm e^{\alpha\n{u}^{\frac{n}{n-s}}}\dx\leq C_{\Om}, \hspace{6mm}\forall\hspace{2mm}u\in\WospOm\hspace{2mm}\mathrm{and}\hspace{2mm}\alpha\in[0,\alpha_*).
\end{equation}
Here, $\alpha_*,C_{\Om}>0$ are constants depending on $n,p,s,\Om$ and $\WospOm$, as a proper subspace of $\WspOm$, is the completion of the set $\CclOm$ with respect to the semi-norm $\br{u}_{s,p,\Rn}$, which is equivalent to \eqref{Eq1} in $\Rn$ on this occasion by Brasco, Lindgren and Parini \cite[Lemma 2.4]{BLP}.

\noindent {\bf(3.)} When $sp>n$, then we have, with $\gamma^*_s:=\frac{sp-n}{p}$,
\begin{equation}\label{Eq6}
\MsqpOm=\WspOm\to\CogOm\hspace{2mm}\mathrm{for}\hspace{2mm}0\leq\gamma\leq\gamma^*_s\hspace{2mm}\mathrm{and}\hspace{2mm}1\leq q\leq\infty.
\end{equation}

Further, via \cite[Theorem 4.54]{DD} and \cite[Theorem 7.1]{DPV}, one easily verifies the result below.

\begin{prop}\label{P1}
Assume $n\geq2$, $p,q\in\br{1,\infty}$ and $s\in\pr{0,1}$.
When $sp<n$, then the embedding $\iota:\MsqpOm\to\LrOm$ is continuous if $1\leq r\leq\max\bre{p^*_s,q}$ and compact if $1\leq r<\max\bre{p^*_s,q}$.
When $sp\geq n$, then the embedding $\iota:\MsqpOm\hookrightarrow\LrOm$ is compact if $1\leq r<\infty$.
\end{prop}

\begin{proof}
Via \eqref{Eq4} and \eqref{Eq6}, we only need to consider the case $sp<n$; this follows almost identically from the proof of \cite[Theorem 7.1]{DPV} seeing \eqref{Eq2}-\eqref{Eq3}.
As a matter of fact, for any cube $\mathbf{Q}$ containing $\Om$, one has for each $u\in\MsqpOm$, recalling $\Om$ is an extension domain,
\begin{equation}
\nm{\tilde{u}}_{\WspCb}\leq\nm{\tilde{u}}_{\WspRn}\leq C'_{\Om}\nm{u}_{\WspOm}\leq C'_{p,q}\nm{u}_{\MsqpOm}.\nonumber
\end{equation}
Here, $\tilde{u}$ denotes the extension of $u$ to $\WspRn$, and $C'_{\Om},C'_{p,q}>0$ are constants depending on $n,p,q,s,\Om$.
Therefore, we can follow \cite[Theorem 7.1]{DPV} to complete the proof.
\end{proof}

Recall for $sp>n$, one actually has
\begin{equation}\label{Eq7}
\MsqpOm\hookrightarrow\CogOm\hspace{2mm}\mathrm{for}\hspace{2mm}0\leq\gamma<\gamma^*_s\hspace{2mm}\mathrm{and}\hspace{2mm}1\leq q\leq\infty.
\end{equation}

\vskip 6pt
\noindent{\bf(II.) $\Om=\Rn$.}
\vskip 6pt

First, note that $\MsqpRn$ represents the Banach space as the completion of the set $\CclRn$ with respect to the norm \eqref{Eq1}.
Now, following Lieb and Loss \cite[Sections 3.2 and 4.3]{LL}, a function $u\in\LlocRn$ is said to {\sl vanish at infinity} provided that $\mathfrak{L}\pr{\bre{x\in\Rn:\n{u(x)}\geq c}}<\infty$ for each positive constant $c>0$, with $\mathfrak{L}$ being the Lebesgue measure.

\noindent {\bf(1.)} When $sp<n$, denote by $\DspRn$ the space of functions $u\in\LlocRn$, where $u$ vanish at infinity and $\br{u}_{s,p,\Rn}<\infty$.
Then, a careful check of Lemma 6.3 and Theorem 6.5 in \cite{DPV} implies $u\in\LpcRn$, and one has the fractional {\sf Sobolev inequality}, saying that
\begin{equation}\label{Eq8}
\Bigl(\IRn\n{u}^{p^*_s}\dx\Bigr)^{\frac{n-sp}{n}}\leq C_1\IIRn\usp\dxy,\hspace{6mm}\forall\hspace{2mm}u\in\DspRn.
\end{equation}
Here, $C_1>0$ is an absolute constant depending on $n,p,s$ with $1\leq p<\infty$.

This leads to $\DspRn=\MspcRn${\footnote {Equivalently, $\DspRn$ is the completion of the set $\CclRn$ with respect to $\br{\cdot}_{s,p,\Rn}$.}}, and one furthermore has
\begin{equation}\label{Eq9}
M^{s;q_1,p}(\Rn)\to M^{s;q_2,p}(\Rn)\hspace{2mm}\mathrm{for}\hspace{2mm}
\begin{array}{ll}
\mathrm{either}&1\leq q_1\leq q_2\leq p^*_s,\medskip\\
\mathrm{or}&p^*_s\leq q_2\leq q_1<\infty
\end{array}
\end{equation}
using {\sf Chebyshev}'s {\sf inequality} and the standard interpolation inequality.
Notice if one would like to have $q=\infty$ included, then $u$ needs to be compactly supported.

\noindent {\bf(2.)} When $sp=n$, then one has the fractional {\sf Gagliardo-Nirenberg inequality} (see Nguyen and Squassina \cite[Lemma 2.1]{NS}) that says
\begin{equation}\label{Eq10}
\nm{u}_{r,\Rn}\leq C_2\nm{u}^{\theta}_{q,\Rn}\br{u}_{s,p,\Rn}^{1-\theta},\hspace{6mm}\forall\hspace{2mm}u\in\MsqpRn.
\end{equation}
Here, we only cited a tailored version for our purpose with $1\leq q\leq r<\infty$ and $\theta:=\frac{q}{r}\in(0,1]$, and $C_2>0$ is an absolute constant depending on $n,p,q,r,s$.
Therefore, we have
\begin{equation}\label{Eq11}
M^{s;q_1,p}(\Rn)\to M^{s;q_2,p}(\Rn)\hspace{2mm}\mathrm{for}\hspace{2mm}1\leq q_1\leq q_2<\infty.
\end{equation}
Notice here $q=\infty$ doesn't contribute to the embedding result \eqref{Eq11}.

\noindent {\bf(3.)} When $sp>n$, a careful reading of Theorem 8.2 in \cite{DPV}, especially their Estimates (8.3) and (8.9), suggests $\MsqpRn\to\CogcRn$ so that $\MsqpRn\to\LiRn$.
Thus, one has
\begin{equation}\label{Eq12}
M^{s;q_1,p}(\Rn)\to M^{s;q_2,p}(\Rn)\hspace{2mm}\mathrm{for}\hspace{2mm}1\leq q_1\leq q_2\leq\infty.
\end{equation}
Notice $\MsqiRn\to\CosRn$ that is consistent with the observation shown at \cite[Page 565]{DPV}, while $\MsipRn$ is the largest possible space for fixed $p,s$ in \eqref{Eq12}.

\begin{prop}\label{P2}
Let $n\geq2$, $p\in\br{1,\infty}$, $q\in[1,\infty)$ and $s\in\pr{0,1}$.
If $sp<n$, then the embedding $\iota:\MsqpRn\to\LrRn$ is continuous for $\min\bre{p^*_s,q}\leq r\leq\max\bre{p^*_s,q}$.
If $sp=n$ or $sp>n$, then the embedding $\iota:\MsqpRn\to\LrRn$ is continuous for $q\leq r<\infty$ or $q\leq r\leq\infty$.
\end{prop}

\vskip 6pt
\noindent{\bf(III.) A fractional Moser-Trudinger type inequality on $\MsqpRn$ when $sp=n$.}
\vskip 6pt

In this section, we employ the preceding result in \cite{PR} to characterize a possible extension of \eqref{Eq5} to $\Rn$ concerning functions $u\in\MsqpRn$ when $sp=n$.
The proof presented here follows basically from Ruf \cite{Ru} and do \'{O} \cite{dO} (see also Li and Ruf \cite{LR}, and Iula \cite{Iu} among many other important contributions on $\Rn$, which cannot be exhaustively listed here).

We now recall some essential facts about {\sf Schwarz symmetrization} by Berestycki and Lions \cite[Appendix III]{BL}, Lieb and Loss \cite[Section 3.3]{LL}, and Beckner \cite[Theorem 3]{Be}.
Let $f(x)$ be a Borel measurable function vanishing at infinity in the sense of Lieb and Loss, and let $f^*(x)(\geq0)$ be the Schwarz symmetrization or spherical rearrangement of $f(x)$ such that $\mathfrak{L}\pr{\bre{x\in\Rn:\n{f(x)}\geq c}}=\mathfrak{L}\pr{\bre{x\in\Rn:f^*(x)\geq c}}$ for all positive constants $c>0$.
Notice $f^*(x)$ is unique, radial, decreasing in $\n{x}$, and lower semi-continuous (so measurable).
Also, we have
\begin{equation}\label{Eq13}
\IRn\Phi(f^*)\dx=\IRn\Phi(\n{f})\dx
\end{equation}
for all continuous functions $\Phi$ with $\Phi(\n{f})$ integrable; so, $\nm{f^*}_{q,\Rn}=\nm{f}_{q,\Rn}$ when $f\in\LqRn$ for $1\leq q\leq\infty$.
Besides, one has
\begin{equation}\label{Eq14}
\IIRn\frac{\n{f^*(x)-f^*(y)}^p}{\n{x-y}^{n+sp}}\dxy\leq\IIRn\frac{\n{f(x)-f(y)}^p}{\n{x-y}^{n+sp}}\dxy.
\end{equation}
For $f^*(x)\in\LqRn$, {\sf Berestycki-Lions}' {\sf radial lemma} \cite[Lemma A.IV]{BL} leads to
\begin{equation}\label{Eq15}
f^*(x)\leq\n{x}^{-\frac{n}{q}}\Bigl(\frac{n}{\on}\Bigr)^{\frac{1}{q}}\nm{f^*}_{q,\Rn}
\end{equation}
for all $x\neq0$, where $\on$ presents the surface area of the unit sphere in $\Rn$.


We next prove a result following \cite[Lemma 1]{dO} and \cite[Proposition 2.1]{Ru} via the $\MsqpRn$-norm, which may be viewed as a fractional {\sf Moser-Trudinger inequality} on $\MsqpRn$.

\begin{thm}\label{T3}
Assume $n\geq2$, $1\leq q<\infty$, $0<s<1$, $sp=n$, $0\leq\alpha<\alpha_*$, and $u\in\MsqpRn$.
Then, there is an absolute constant $C(\alpha,q,s)>0$ depending on $\alpha,n,q,s$ such that
\begin{equation}\label{Eq16}
\sup_{\nm{u}_{\MsqpRn}\leq1}\IRn\Psi_{\alpha,q,s}(\n{u})\dx\leq C(\alpha,q,s).
\end{equation}
Here, for $v\geq0$ and the least positive integer $\beta_0$ with $\frac{\beta_0n}{n-s}\geq q\geq1$, $\Psi_{\alpha,q,s}(v):=\sum\limits_{\ell=\beta_0}^{\infty}\frac{\alpha^{\ell}}{\ell!}v^{\frac{\ell n}{n-s}}$.
\end{thm}

\begin{proof}
To save notation, assume without loss of generality $u\geq0$.
Via $u^*$, and seeing \eqref{Eq13}, \eqref{Eq14} and $\Psi_{\alpha,q,s}(u^*)\geq0$, we can simply consider $u=u^*$ and decompose
\begin{equation}
\IRn\Psi_{\alpha,q,s}(u)\dx=\int_{\BRo}\Psi_{\alpha,q,s}(u)\dx+\int_{\BRco}\Psi_{\alpha,q,s}(u)\dx.\nonumber
\end{equation}
Here and hereafter, $\BR$ is the ball of radius $R$ in $\Rn$ centered at the origin and $\BRc:=\Rn\setminus\BR$, while $R_0>0$ is a sufficiently large absolute constant to be determined later.

To estimate the integral over $\BRo$, write $v(x):=\max\bre{u(\n{x})-u_0,0}$ for $u_0:=u(R_0)>0$, a constant.
Then, $v\equiv0$ on $\BRco$ and $v\in\WospBRo$, since by Lemma 2.4 of \cite{BLP}
\begin{equation}\label{Eq17}
\begin{split}
&\frac{\Vn}{\pr{36R_0}^n}\int_{\BRo}v^p(x)\dx\leq\int_{\BRo}\int_{\mathbf{B}_{R_0}(x_0)}\frac{\n{v(x)-v(y)}^p}{\n{x-y}^{2n}}\dxy\\
\leq&\int_{\BRo}\int_{\BRo}\frac{\n{v(x)-v(y)}^p}{\n{x-y}^{2n}}\dxy+2\int_{\BRo}\int_{\BRco}\frac{\n{v(x)-v(y)}^p}{\n{x-y}^{2n}}\dxy\\
=&\int_{\BRo}\int_{\BRo}\frac{\n{u(\n{x})-u(\n{y})}^p}{\n{x-y}^{2n}}\dxy+2\int_{\BRo}\int_{\BRco}\frac{\n{u(\n{x})-u_0}^p}{\n{x-y}^{2n}}\dxy
\end{split}
\end{equation}
follows; seeing $u(\n{y})\leq u_0$ when $\n{y}\geq R_0$, one has $\frac{\Vn}{\pr{36R_0}^n}\nm{v}^p_{p,\BRo}\leq\br{v}^p_{s,p,\Rn}\leq\br{u}^p_{s,p,\Rn}\leq1$.
Here, $\Vn$ presents the volume of the unit ball in $\Rn$, and $\mathbf{B}_{R_0}(x_0)$ denotes a ball of radius $R_0$ in $\BRco$ centered at some point $x_0$ with $\n{x_0}=4R_0$.

Notice \cite[Estimate (5)]{dO} yields $\pr{1+t}^{\frac{n}{n-s}}\leq t^{\frac{n}{n-s}}+K_1t^{\frac{s}{n-s}}+1$ for all $t\in\pr{0,\infty}$ with $K_1>0$ an absolute constant.
Actually, for $f(t):=\frac{\pr{1+t}^{\frac{n}{n-s}}-t^{\frac{n}{n-s}}-1}{t^{\frac{s}{n-s}}}$ that is continuous on $\pr{0,\infty}$, one sees $\lim\limits_{t\to0}f(t)=0$ as $n\geq2>2s$, while using $\tilde{t}:=\frac{1}{t}$ one sees $\lim\limits_{t\to\infty}f(t)=\frac{n}{n-s}$.

Next, for $t=\frac{v(x)}{u_0}$, it follows that $u^{\frac{n}{n-s}}\leq\bre{v+u_0}^{\frac{n}{n-s}}\leq v^{\frac{n}{n-s}}+K_1v^{\frac{s}{n-s}}u_0+u_0^{\frac{n}{n-s}}$.
Apply {\sf Young}'s {\sf inequality} to $K_1v^{\frac{s}{n-s}}u_0$ to observe, for an absolute constant $K_2>0$,
\begin{equation}\label{Eq18}
u^{\frac{n}{n-s}}\leq v^{\frac{n}{n-s}}\pr{1+u^{\frac{n}{s}}_0}+K_2+u_0^{\frac{n}{n-s}}.
\end{equation}
Define $\tilde{u}:=v\pr{1+u^{\frac{n}{s}}_0}^{\frac{n-s}{n}}\in\WospBRo$ and notice $\br{\tilde{u}}_{s,p,\Rn}\leq\pr{1+u^{\frac{n}{s}}_0}^{\frac{n-s}{n}}\br{v}_{s,p,\Rn}$, so that by the discussions after \eqref{Eq17} it leads to $\br{\tilde{u}}^p_{s,p,\Rn}\leq\pr{1+u^{\frac{n}{s}}_0}^{\frac{n-s}{s}}\pr{1-\nm{u}_{q,\Rn}}^{\frac{n}{s}}$.

From \eqref{Eq15}, $u^{\frac{n}{s}}_0\leq\sigma\nm{u}^{\frac{n}{s}}_{q,\Rn}$ for $\sigma:=\bigl(\frac{n}{\on R^n_0}\bigr)^{\frac{n}{qs}}$.
Write $g(t):=\pr{1+\sigma t^{\frac{n}{s}}}^{\frac{n-s}{s}}\pr{1-t}^{\frac{n}{s}}-1$ and take $g'(t)=\frac{n}{s}\pr{1+\sigma t^{\frac{n}{s}}}^{\frac{n-2s}{s}}\pr{1-t}^{\frac{n-s}{s}}\bigl(\frac{n-s}{s}\sigma t^{\frac{n-s}{s}}-\frac{n}{s}\sigma t^{\frac{n}{s}}-1\bigr)$, with $g(0)=0$ and $g(1)=-1$.
Define $h(t):=\frac{n-s}{s}\sigma t^{\frac{n-s}{s}}-\frac{n}{s}\sigma t^{\frac{n}{s}}-1$ with $h'(t_0)=0$ for $t_0:=\bigl(\frac{n-s}{n}\bigr)^2$.
So, when $R_0$ is so large that $\sigma\leq\bigl(\frac{n}{n-s}\bigr)^{\frac{2n-s}{s}}$, then one sees $h(t)\leq0$ and hence $g'(t)\leq0$ on $\br{0,1}$.
That is, $g(t)\leq0$ for $0\leq t\leq1$, which clearly implies $\br{\tilde{u}}_{s,p,\Rn}\leq1$ provided $R^n_0\geq\frac{n}{\on}\bigl(\frac{n-s}{n}\bigr)^{\frac{q\pr{2n-s}}{n}}$.

As a result, one can combine this with \eqref{Eq5} and \eqref{Eq18} to deduce
\begin{equation}
\int_{\BRo}\Psi_{\alpha,q,s}(u)\dx\leq\int_{\BRo}e^{\alpha u^{\frac{n}{n-s}}}
\leq e^{\alpha\bigl(K_2+u_0^{\frac{n}{n-s}}\bigr)}\int_{\BRo}e^{\alpha\tilde{u}^{\frac{n}{n-s}}}\dx\leq C_1(\alpha,q,s).\nonumber
\end{equation}
Here, $C_1(\alpha,q,s)>0$ is an absolute constant depending on $\alpha,n,q,s$.

To estimate the integral over $\BRco$, one applies \eqref{Eq15} with $\nm{u}_{q,\Rn}\leq1$ to observe
\begin{equation}
\int_{\BRco}u^{\frac{\ell n}{n-s}}\dx\leq\on\Bigl(\frac{n}{\on}\Bigr)^{\frac{\ell n}{q\pr{n-s}}}\int^{\infty}_{R_0}t^{-\frac{\ell n^2}{q\pr{n-s}}+n-1}\dt.\nonumber
\end{equation}
For $\frac{\beta_0n}{n-s}\geq q$, one employs Proposition \ref{P2} to the starting term in $\Psi_{\alpha,q,s}(u)$ and thus
\begin{equation}
\sum_{\ell=\beta_0}^{\infty}\frac{\alpha^{\ell}}{\ell!}\int_{\BRco}u^{\frac{\ell n}{n-s}}\dx\leq\frac{\alpha^{\beta_0}}{\beta_0!}\int_{\BRco}u^{\frac{\beta_0n}{n-s}}\dx
+\frac{q\pr{n-s}}{n^2}R^n_0\on e^{\alpha\sigma^{\frac{s}{n-s}}}\leq C_2(\alpha,q,s).\nonumber
\end{equation}
Here, $C_2(\alpha,q,s)>0$ is an absolute constant depending on $\alpha,n,q,s$.

Accordingly, \eqref{Eq16} follows for $C(\alpha,q,s):=\max\bre{C_1(\alpha,q,s),C_2(\alpha,q,s)}>0$.
\end{proof}

\vskip 6pt
\noindent{\bf(IV.) Some compact embedding results regarding $\MVsqpRn\hookrightarrow\LrKRn$.}
\vskip 6pt

Let $V(x)>0$ be a Lebesgue measurable function in $\Rn$ such that $\inf\limits_{\mathbf{D}}V(x)\geq V_{\mathbf{D}}>0$ for all compact subsets $\mathbf{D}\Subset\Rn$.
Take $n\geq2$, $p\in\br{1,\infty}$, $q\in[1,\infty)$, and $s\in\pr{0,1}$.
When $sp<n$, we designate $\MVsqpRn$ to be the Banach space as the completion of the set $\CclRn$ with respect to the norm $\nm{u}_{\MVsqpRn}:=\nm{u}_{\LqVRn}+\br{u}_{s,p,\Rn}$ where $\nm{u}^q_{\LqVRn}:=\IRn\n{u}^qV\dx$, so that $\MVsqpRn$ is a subspace of $\DspRn$.
When $sp\geq n$, we further require $\inf\limits_{\Rn}V(x)\geq V_0>0$ and define $\MVsqpRn$ similarly, so that $\MVsqpRn$ is a subspace of $\MsqpRn$.

One recalls that some closely related results on $\WlpRn$ may be found in Chiappinelli \cite[Theorem 1]{Ch}, Schneider \cite[Theorem 2.3]{Sc}, and Bonheure and Van Schaftingen \cite[Theorem 4]{BVS}.
Yet, it seems that the results presented here are not available even on $\WspRn$.

\noindent {\bf(1.)} When $sp<n$, we can prove some compact embedding results listed below.

\begin{thm}\label{T4}
Assume $n\geq2$, $1\leq p\leq\infty$, $0<s<1$ with $sp<n$, and $1\leq q\leq r<p^*_s$.
Let $K(x),V(x)>0:\Rn\to\R$ be two Lebesgue measurable functions such that $K(x)\in\LtOm$ for some $t\in\bigl(\frac{p^*_s}{p^*_s-r},\infty\bigr]$ on each set $\Om$ of $\Rn$ with $\mathfrak{L}\pr{\Om}<\infty$ and $K(x)V^{-\tau}(x)$ vanishes at infinity for $\tau:=\frac{p^*_s-r}{p^*_s-q}$.
Then, the embedding $\MVsqpRn\hookrightarrow\LrKRn$ is compact.
\end{thm}

\begin{proof}
Without loss of generality, assume $\bre{u_l:l\geq1}$ is a sequence of functions in $\MVsqpRn$ with $u_l\rightharpoonup0$ when $l\to\infty$ and $\nm{u_l}_{\MVsqpRn}$ uniformly bounded.
For each $\epsilon>0$, set $\Wc:=\bigl\{x\in\Rn:K(x)V^{-\tau}(x)\geq\epsilon\bigr\}$ and $\Wc^c:=\Rn\setminus\Wc$ to decompose
\begin{equation}\label{Eq19}
\IRn\n{u_l}^rK\dx=\int_{\Wc}\n{u_l}^rK\dx+\int_{\Wc^c}\n{u_l}^rK\dx.
\end{equation}

For the integral over $\Wc^c$, as $\frac{r-q\tau}{p^*_s}=1-\tau=\frac{r-q}{p^*_s-q}\in[0,1)$, it is easy to derive
\begin{equation}\label{Eq20}
\begin{split}
&\int_{\Wc^c}\n{u_l}^r\frac{K}{V^{\tau}}V^{\tau}\dx\leq\epsilon\IRn\n{u_l}^rV^{\tau}\dx
\leq\epsilon\Bigl(\IRn\n{u_l}^qV\dx\Bigr)^{\tau}\Bigl(\IRn\n{u_l}^{p^*_s}\dx\Bigr)^{\frac{r-q\tau}{p^*_s}}\\
\leq\,&\epsilon C'_1\Bigl(\IRn\n{u_l}^qV\dx\Bigr)^{\tau}\Bigl(\IIRn\ulsp\dxy\Bigr)^{\frac{r-q\tau}{p}}\leq\epsilon C'_1\nm{u_l}^r_{\MVsqpRn}
\end{split}
\end{equation}
in view of \eqref{Eq8}, with $C'_1>0$ an absolute constant independent of $u_l$ for any $l\geq1$.

For the integral over $\Wc$, noticing $\mathfrak{L}\pr{\Wc}<\infty$ and the fact that $\inf\limits_{\BR}V(x)\geq V_{\BR}>0$ yields the embedding $\MVsqpBR\to\MsqpBR\hookrightarrow L^{\frac{rt}{t-1}}(\BR)$ with $\frac{rt}{t-1}\in[r,p^*_s)$, we deduce
\begin{equation}\label{Eq21}
\begin{split}
&\int_{\Wc}\n{u_l}^rK\dx=\int_{\Wc\cap\BR}\n{u_l}^rK\dx+\int_{\Wc\cap\BR^c}\n{u_l}^rK\dx\\
\leq&\Bigl(\int_{\Wc}K^t\dx\Bigr)^{\frac{1}{t}}\Biggl\{\Bigl(\IBR\n{u_l}^{\frac{rt}{t-1}}\dx\Bigr)^{\frac{t-1}{t}}\\
&+\Bigl(\IRn\n{u_l}^{p^*_s}\dx\Bigr)^{\frac{r}{p^*_s}}\bre{\mathfrak{L}\pr{\Wc\cap\BR^c}}^{1-\frac{1}{t}-\frac{r}{p^*_s}}\Biggr\}\to0
\end{split}
\end{equation}
as $R\to\infty$ and $l\to\infty$ for a subsequence of $\bre{u_l:l\geq1}$ using the same notation.

Notice the embedding $\MVsqpRn\to\LrKRn$ is continuous provided $\frac{p^*_s}{p^*_s-r}\leq t\leq\infty$.
Hence, plugging \eqref{Eq20} and \eqref{Eq21} altogether back to \eqref{Eq19} finishes our proof completely.
\end{proof}

\begin{thm}\label{T5}
Assume $n\geq2$, $1\leq p\leq\infty$, $0<s<1$ with $sp<n$, and $1\leq p^*_s\leq r<q$.
Let $K(x),V(x)>0:\Rn\to\R$ be such that $K(x)\in\LtOm$ for some $t\in\bigl(\frac{q}{q-r},\infty\bigr]$ on each set $\Om$ of $\Rn$ with $\mathfrak{L}\pr{\Om}<\infty$, $\inf\limits_{\Rn}V(x)\geq V_0>0$, and $K(x)V^{-\tau}(x)$ vanishes at infinity for $\tau:=\frac{r-p^*_s}{q-p^*_s}$.
Then, the embedding $\MVsqpRn\hookrightarrow\LrKRn$ is compact.
\end{thm}

\begin{proof}
The proof is a minor modification of Theorem \ref{T4} using the embeddings $\MVsqpRn\to\MsqpRn\to\LqRn$ and $\MVsqpBR\to\MsqpBR\hookrightarrow L^{\frac{rt}{t-1}}(\BR)$ for $\frac{rt}{t-1}\in[r,q)$.
There is no need for change in \eqref{Eq20} while the only change in \eqref{Eq21} is replacing $p^*_s$ by $q$.
One also notices the embedding $\MVsqpRn\to\LrKRn$ is continuous provided $\frac{q}{q-r}\leq t\leq\infty$.
\end{proof}

\begin{thm}\label{T6}
Assume $n\geq2$, $1\leq p\leq\infty$, $0<s<1$ with $sp<n$, and $1\leq p^*_s\leq r<q$.
Let $K(x),V(x)>0:\Rn\to\R$ satisfy $K(x)\in\LtlocRn$ for some $t\in\bigl(\frac{q}{q-r},\infty\bigr]$ and $K(x)V^{-\tau}(x)\to0$ uniformly for $\tau:=\frac{r-p^*_s}{q-p^*_s}$.
Then, the embedding $\MVsqpRn\hookrightarrow\LrKRn$ is compact.
\end{thm}

\begin{proof}
Assume $\bre{u_l:l\geq1}$ is a sequence of functions in $\MVsqpRn$ with $u_l\rightharpoonup0$ when $l\to\infty$ and $\nm{u_l}_{\MVsqpRn}$ uniformly bounded.
Then, one can decompose
\begin{equation}\label{Eq22}
\IRn\n{u_l}^rK\dx=\IBR\n{u_l}^rK\dx+\IBRc\n{u_l}^rK\dx.
\end{equation}
For the integral over $\BR$, seeing that $\inf\limits_{\BR}V(x)\geq V_{\BR}>0$ yields the embedding $\MVsqpBR\to\MsqpBR\hookrightarrow L^{\frac{rt}{t-1}}(\BR)$ with $\frac{rt}{t-1}\in[r,q)$, one deduces
\begin{equation}\label{Eq23}
\IBR\n{u_l}^rK\dx\leq\Bigl(\IBR K^t\dx\Bigr)^{\frac{1}{t}}\Bigl(\IBR\n{u_l}^{\frac{rt}{t-1}}\dx\Bigr)^{\frac{t-1}{t}}\to0
\end{equation}
as $l\to\infty$ for a subsequence of $\bre{u_l:l\geq1}$ using the same notation.
For the integral over $\BRc$, one deduces, similar to \eqref{Eq20},
\begin{equation}
\begin{split}
&\IBRc\n{u_l}^rK\dx\leq\nm{KV^{-\tau}}_{\infty,\BRc}\Bigl(\IRn\n{u_l}^qV\dx\Bigr)^{\tau}\Bigl(\IRn\n{u_l}^{p^*_s}\dx\Bigr)^{\frac{r-q\tau}{p^*_s}}\\
\leq\,&C'_1\nm{KV^{-\tau}}_{\infty,\BRc}\nm{u_l}^r_{\MVsqpRn}\to0\nonumber
\end{split}
\end{equation}
as $R\to\infty$.
So, $u_l\to0$ in $\LrKRn$ for a subsequence relabeled with the same index $l$.

Notice the embedding $\MVsqpRn\to\LrKRn$ is continuous provided $K(x)\in\LtlocRn$ for $\frac{q}{q-r}\leq t\leq\infty$ and $K(x)V^{-\tau}(x)$ is eventually bounded at infinity for $\tau=\frac{r-p^*_s}{q-p^*_s}$.
\end{proof}

Finally, let's consider the case where $1\leq r<\min\bre{p^*_s,q}$ uncovered in the preceding results.

\begin{thm}\label{T7}
Assume $n\geq2$, $1\leq p\leq\infty$, $0<s<1$ with $sp<n$, and $1\leq r<\min\bre{p^*_s,q}\leq\max\bre{p^*_s,q}<\infty$.
Let $K(x),V(x)>0$ satisfy $K(x)\in\LtlocRn$ for some $t\in\bigl(\frac{\max\bre{p^*_s,q}}{\max\bre{p^*_s,q}-r},\infty\bigr]$ and $K^{\psi_{p^*_s}(\tau)}(x)V^{-\tau}(x)\in\LlRn$ for $\psi_{p^*_s}(\tau):=\frac{p^*_s+\tau(p^*_s-q)}{p^*_s-r}$ and some $\tau\in\bigl[0,\frac{r}{q-r}\bigr]$.
Then, the embedding $\MVsqpRn\hookrightarrow\LrKRn$ is compact.
\end{thm}

\begin{proof}
Write $\mathfrak{x}=\frac{\tau(p^*_s-r)}{p^*_s+\tau(p^*_s-q)}$, $\mathfrak{y}=\frac{r+\tau(r-q)}{p^*_s+\tau(p^*_s-q)}$ and $\mathfrak{z}=\frac{p^*_s-r}{p^*_s+\tau(p^*_s-q)}$ with $\mathfrak{x}+\mathfrak{y}+\mathfrak{z}=1$.
Now, set $r_1=\mathfrak{x}^{-1}$, $r_2=\mathfrak{y}^{-1}$ and $r_3=\mathfrak{z}^{-1}$ to see, for $u\in\MVsqpRn$ on all domains $\Om$ in $\Rn$,
\begin{equation}\label{Eq24}
\begin{split}
&\IOm\n{u}^r\frac{K}{V^{\mathfrak{x}}}V^{\mathfrak{x}}\dx
\leq\Bigl(\IOm\n{u}^qV\dx\Bigr)^{\frac{1}{r_1}}\Bigl(\IOm\n{u}^{p^*_s}\dx\Bigr)^{\frac{1}{r_2}}\Bigl(\IOm K^{r_3}V^{-r_3\mathfrak{x}}\dx\Bigr)^{\frac{1}{r_3}}\\
\leq&\,C'_1\Bigl(\IOm\n{u}^qV\dx\Bigr)^{\frac{1}{r_1}}\Bigl(\IIRn\usp\dxy\Bigr)^{\frac{p^*_s}{pr_2}}\Bigl(\IOm K^{\psi_{p^*_s}(\tau)}V^{-\tau}\dx\Bigr)^{\frac{1}{r_3}}
\end{split}
\end{equation}
by {\sf H\"{o}lder}'s {\sf inequality} and \eqref{Eq8} with $p^*_s=r_2(r-q\mathfrak{x})$, provided $\mathfrak{x},\mathfrak{y},\mathfrak{z}\in\pr{0,1}$.

To have $\mathfrak{x}<1$, one sees $0<\tau<\frac{p^*_s}{q-r}$ when $q\leq p^*_s$, and either $0<\tau<\frac{p^*_s}{q-r}$ or $\frac{p^*_s}{q-p^*_s}<\tau<\infty$ when $q>p^*_s$.
To have $\mathfrak{x}>0$, one sees $0<\tau<\infty$ when $q\leq p^*_s$, and $0<\tau<\frac{p^*_s}{q-p^*_s}$ when $q>p^*_s$.

To have $\mathfrak{y}>0$, one sees $0<\tau<\frac{r}{q-r}$ when $q\leq p^*_s$, and either $0<\tau<\frac{r}{q-r}$ or $\frac{p^*_s}{q-p^*_s}<\tau<\infty$ when $q>p^*_s$.
To have $\mathfrak{y}<1$, one sees $0<\tau<\infty$ when $q\leq p^*_s$, and $0<\tau<\frac{p^*_s}{q-p^*_s}$ when $q>p^*_s$.

Notice $\mathfrak{x}=\tau\mathfrak{z}$.
To have $\mathfrak{z}<1$, one sees $0<\tau<\infty$ when $q\leq p^*_s$, and either $0<\tau<\frac{r}{q-p^*_s}$ or $\frac{p^*_s}{q-p^*_s}<\tau<\infty$ when $q>p^*_s$.

Summarize the above analyses to conclude that $\tau\in\bigl(0,\frac{r}{q-r}\bigr)$.

We certainly can take $\mathfrak{y}=0$ and see $\tau=\frac{r}{q-r}$ without having $p^*_s$ involved, and may also take $\mathfrak{x}=\tau=0$ corresponding to $\MVsqpRn=\DspRn$ with $V(x)\equiv0$.

Next, one has $\frac{q}{r_1}+\frac{p^*_s}{r_2}=q\mathfrak{x}+p^*_s\mathfrak{y}=r$ and thus the embedding $\MVsqpRn\to\LrKRn$, when replacing $\Om$ in \eqref{Eq24} by $\Rn$, is continuous (even for $\frac{\max\bre{p^*_s,q}}{\max\bre{p^*_s,q}-r}\leq t\leq\infty$).

Finally, using the same $\bre{u_l:l\geq1}$, one sees the decomposition \eqref{Eq22} and the estimate \eqref{Eq23}, because $\MVsqpBR\to\MsqpBR\hookrightarrow L^{\frac{rt}{t-1}}(\BR)$ with $\frac{rt}{t-1}\in[r,\max\bre{p^*_s,q})$.
For the integral over $\BRc$, we apply \eqref{Eq24} on $\Om=\BRc$ to derive, as $R\to\infty$,
\begin{equation}
\IBRc\n{u_l}^rK\dx\leq C'_1\nm{K^{\psi_{p^*_s}(\tau)}V^{-\tau}}^{\psi^{-1}_{p^*_s}(\tau)}_{1,\BRc}\nm{u_l}^r_{\MVsqpRn}\to0.\nonumber
\end{equation}
Thus, $u_l\to0$ in $\LrKRn$ for a subsequence relabeled with the same index $l$.
\end{proof}

Notice in Theorems \ref{T4} and \ref{T6}, we do not need $\inf\limits_{\Rn}V(x)\geq V_0>0$, which appeared in Theorem \ref{T5}.
This condition can be replaced by an integrability condition as in Theorem \ref{T7}.

\begin{cor}\label{C8}
Let $n\geq2$, $p\in\br{1,\infty}$, $s\in\pr{0,1}$ with $sp<n$, $1\leq p^*_s\leq r<q$, and $K(x),V(x)>0:\Rn\to\R$ with $K(x)\in\LtlocRn$ for some $t\in\bigl(\frac{q}{q-r},\infty\bigr]$ and $K^{\psi_{p^*_s}(\tau)}(x)V^{-\tau}(x)\in\LlRn$ for some $\tau\in\bigl[\frac{r}{q-r},\infty\bigr)$.
Then, the embedding $\MVsqpRn\hookrightarrow\LrKRn$ is compact.
\end{cor}

\begin{proof}
The proof is essentially the same as that of Theorem \ref{T7}.
One keeps in mind $q>r\geq p^*_s$ when ensuring $\mathfrak{x},\mathfrak{y},\mathfrak{z}\in\pr{0,1}$.
In fact, to see $\mathfrak{x}>0$, one has $\tau>\frac{p^*_s}{q-p^*_s}$ and to see $\mathfrak{x}<1$, one has $\tau>\frac{p^*_s}{q-r}$; to see $\mathfrak{y}>0$, one has $\tau>\frac{r}{q-r}$ and to see $\mathfrak{y}<1$, one has $\tau>\frac{p^*_s}{q-p^*_s}$; to see $\mathfrak{z}<1$, one has $\tau>\frac{r}{q-p^*_s}$.
Summarizing these leads to $\tau\in\bigl(\frac{r}{q-r},\infty\bigr)$ while $\mathfrak{y}=0$ implies $\tau=\frac{r}{q-r}$ without involving $p^*_s$.
Notice $\MVsqpRn\to\LrKRn$ is continuous for $\frac{q}{q-r}\leq t\leq\infty$.
\end{proof}

\noindent {\bf(2.)} When $sp\geq n$, we similarly have the compact embedding results as follows.

Recall the fractional {\sf Gagliardo-Nirenberg inequality} in \cite[Lemma 2.1]{NS} says that
\begin{equation}\label{Eq25}
\nm{u}_{r,\Rn}\leq C_2\nm{u}^{\theta}_{q,\Rn}\br{u}_{s,p,\Rn}^{1-\theta},\hspace{6mm}\forall\hspace{2mm}u\in\MsqpRn.
\end{equation}
Here, $n\geq2$, $1<p\leq\infty$, $1\leq q<\infty$, $0<s<1$, $\frac{1}{r}=(1-\theta)\frac{n-sp}{np}+\theta\frac{1}{q}$ with $\theta\in\br{0,1}$, and $r$ lies in between $q$ and $p^*_s$ when $sp<n$ while $q\leq r<\infty$ when $sp\geq n$, with $C_2>0$ an absolute constant depending on $n,p,q,r,s$.

\begin{thm}\label{T9}
Assume $n\geq2$, $1\leq p\leq\infty$, $0<s<1$ with $sp\geq n$, and $1\leq q\leq r<\infty$.
Let $K(x),V(x)>0:\Rn\to\R$ be such that $K(x)\in\LtOm$ for some $t\in(1,\infty]$ on each set $\Om$ of $\Rn$ with $\mathfrak{L}\pr{\Om}<\infty$, $\inf\limits_{\Rn}V(x)\geq V_0>0$, and $K(x)V^{-\tau}(x)$ vanishes at infinity for some $\tau\in\left(0,1\right)$ if $q<r$ or $\tau=1$ if $q=r$.
Then, the embedding $\MVsqpRn\hookrightarrow\LrKRn$ is compact.
\end{thm}

\begin{proof}
First, note the continuous embedding $\MVsqpRn\to\MsqpRn$.
Employing the same notations as in Theorem \ref{T4}, and seeing $\frac{r-q\tau}{1-\tau}\geq q$ and \eqref{Eq19}, one derives, like \eqref{Eq20},
\begin{equation}
\begin{split}
&\int_{\Wc^c}\n{u_l}^r\frac{K}{V^{\tau}}V^{\tau}\dx\leq\epsilon\IRn\n{u_l}^rV^{\tau}\dx
\leq\epsilon\Bigl(\IRn\n{u_l}^qV\dx\Bigr)^{\tau}\Bigl(\IRn\n{u_l}^{\frac{r-q\tau}{1-\tau}}\dx\Bigr)^{1-\tau}\\
\leq\,&\epsilon C'_2\nm{u_l}^{q\tau}_{\LqVRn}\nm{u_l}^{r-q\tau}_{\MsqpRn}\leq\epsilon C'_2\nm{u_l}^r_{\MVsqpRn}\nonumber
\end{split}
\end{equation}
by virtue of \eqref{Eq25}, with $C'_2>0$ an absolute constant independent of $u_l$ for any $l\geq1$.
Moreover, seeing $\MVsqpBR\to\MsqpBR\hookrightarrow L^{\frac{rt}{t-1}}(\BR)$ with $\frac{rt}{t-1}\in[r,\infty)$, one observes that, for some $\tilde{t}\in\bigl(\frac{t}{t-1},\infty\bigr]$ with $\lim\limits_{\tilde{t}\to\infty}\Bigl(\IRn\n{u_l}^{r\tilde{t}}\dx\Bigr)^{\frac{1}{\tilde{t}}}=\IRn\n{u_l}^{r}\dx$, like \eqref{Eq21},
\begin{equation}
\begin{split}
\int_{\Wc}\n{u_l}^rK\dx\leq&\Bigl(\int_{\Wc}K^t\dx\Bigr)^{\frac{1}{t}}\Biggl\{\Bigl(\IBR\n{u_l}^{\frac{rt}{t-1}}\dx\Bigr)^{\frac{t-1}{t}}\\
&+\Bigl(\IRn\n{u_l}^{r\tilde{t}}\dx\Bigr)^{\frac{1}{\tilde{t}}}\bre{\mathfrak{L}\pr{\Wc\cap\BR^c}}^{1-\frac{1}{t}-\frac{1}{\tilde{t}}}\Biggr\}\to0\nonumber
\end{split}
\end{equation}
as $R\to\infty$ and $l\to\infty$ for a subsequence of $\bre{u_l:l\geq1}$ using the same notation.

Note the embedding $\MVsqpRn\to\LrKRn$ is continuous if $\frac{t}{t-1}\leq\tilde{t}\leq\infty$.
\end{proof}

\begin{thm}\label{T10}
Assume $n\geq2$, $1\leq p\leq\infty$, $0<s<1$ with $sp\geq n$, and $1\leq r<q<\infty$.
Let $K(x),V(x)>0:\Rn\to\R$ satisfy $K(x)\in\LtlocRn$ for some $t\in(1,\infty]$, $\inf\limits_{\Rn}V(x)\geq V_0>0$, and $K^{\psi_{\tilde{q}}(\tau)}(x)V^{-\tau}(x)\in\LlRn$ for $\psi_{\tilde{q}}(\tau):=\frac{\tilde{q}+\tau(\tilde{q}-q)}{\tilde{q}-r}$, some $\tilde{q}\in\pr{q,\infty}$, and some $\tau\in\bigl(0,\frac{r}{q-r}\bigr]$.
Then, the embedding $\MVsqpRn\hookrightarrow\LrKRn$ is compact.
\end{thm}

\begin{proof}
First, note the continuous embedding $\MVsqpRn\to\MsqpRn$.
Employing the same notations as in Theorem \ref{T7} with $\tilde{\mathfrak{x}}=\frac{\tau(\tilde{q}-r)}{\tilde{q}+\tau(\tilde{q}-q)}$, $\tilde{\mathfrak{y}}=\frac{r+\tau(r-q)}{\tilde{q}+\tau(\tilde{q}-q)}$ and $\tilde{\mathfrak{z}}=\frac{\tilde{q}-r}{\tilde{q}+\tau(\tilde{q}-q)}$ for an arbitrarily chosen $\tilde{q}\in\pr{q,\infty}$ plus $\tilde{r}_1=\tilde{\mathfrak{x}}^{-1}$, $\tilde{r}_2=\tilde{\mathfrak{y}}^{-1}$ and $\tilde{r}_3=\tilde{\mathfrak{z}}^{-1}$, and seeing \eqref{Eq25}, one derives, like \eqref{Eq24}, the continuous embedding $\MVsqpRn\to\LrKRn$ since $\frac{q}{\tilde{r}_1}+\frac{\tilde{q}}{\tilde{r}_2}=q\tilde{\mathfrak{x}}+\tilde{q}\tilde{\mathfrak{y}}=r$.

In addition, similar calculations lead to $\tau\in\bigl(0,\frac{r}{q-r}\bigr]$.

Finally, using the same $\bre{u_l:l\geq1}$, one sees the decomposition \eqref{Eq22} and the estimate \eqref{Eq23}, because $\MVsqpBR\to\MsqpBR\hookrightarrow L^{\frac{rt}{t-1}}(\BR)$ with $\frac{rt}{t-1}\in[r,\infty)$.
For the integral over $\BRc$, one analogously has, as $R\to\infty$,
\begin{equation}
\IBRc\n{u_l}^rK\dx\leq C'_2V^{-\frac{\tilde{q}\theta}{q\tilde{r}_2}}_0\nm{K^{\psi_{\tilde{q}}(\tau)}V^{-\tau}}^{\psi^{-1}_{\tilde{q}}(\tau)}_{1,\BRc}\nm{u_l}^r_{\MVsqpRn}\to0.\nonumber
\end{equation}
Thus, $u_l\to0$ in $\LrKRn$ for a subsequence relabeled with the same index $l$.
\end{proof}

\noindent {\bf(3.)} Some applications.

One feels free to set either $V(x)\equiv1$ or $K(x)\equiv1$ to see some compact embedding results, in particular, $\MVsqpRn\hookrightarrow\LqRn$ and $\MVsqpRn\hookrightarrow\LlRn$ for appropriate weight functions $V(x)$; when $p=q$, one has $\WVspRn\hookrightarrow\LpRn$ and $\WVspRn\hookrightarrow\LlRn$.

Finally, following Servadei and Valdinoci \cite[Section 3]{SV} (see also \cite[Section 3.1]{MRS}) or applying Auchmuty \cite{Au}, for the standard fractional Laplacian $\pr{-\Delta}^s$, the eigenvalue problems
\begin{equation}\label{Eq26}
-\pr{-\Delta}^sf(x)+V(x)f(x)=\lambda K(x)f(x)\hspace{2mm}\mathrm{in}\hspace{2mm}\Rn,
\end{equation}
as well as the eigenvalue problems
\begin{equation}\label{Eq27}
-\pr{-\Delta}^sg(x)=\delta\tilde{K}(x)g(x)\hspace{2mm}\mathrm{in}\hspace{2mm}\Rn
\end{equation}
if $sp<n$, possess families of eigenvalues $\bre{\lambda_k>0:k\geq1}$ and $\bre{\delta_l>0:l\geq1}$, and sequences of associated eigenfunctions $\bre{f_k\in\WVstRn:k\geq1}$ and $\bre{g_l\in\DstRn:l\geq1}$, respectively, by virtue of Theorem \ref{T4} (when taking $p=q=r=2$ to work in Hilbert spaces) and Theorem \ref{T7} (when taking $V(x)\equiv0$ and $p=r=2$ to work in Hilbert spaces).

$\lambda_1,\delta_1$ are simple, $\lambda_k,\delta_l$ have finite multiplicities, and $\bre{\lambda_k>0:k\geq1}$ and $\bre{\delta_l>0:l\geq1}$ are increasing with $\lim\limits_{k\to\infty}\lambda_k=\infty$ and $\lim\limits_{l\to\infty}\delta_l=\infty$.
Besides, $f_1,g_1$ are nonnegative in $\Rn$, and the sequences $\bre{f_k\in\WVstRn:k\geq1}$ and $\bre{g_l\in\DstRn:l\geq1}$ provide orthogonal bases to the spaces $\WVstRn,\LtKRn$ and $\DstRn,\LtKtRn$, respectively.

The procedure is standard now, with details left for the interested reader.

\vskip 21pt

\bibliographystyle{amsplain}

\end{document}